\documentclass[a4paper,11pt]{amsart}

\usepackage{amsmath}
\usepackage{amssymb}
\usepackage{amstext}
\usepackage{amsbsy}
\usepackage{amsopn}
\usepackage{amsthm}
\usepackage{amscd}
\usepackage{amsxtra}
\usepackage{amsfonts}

\usepackage{ifthen}
\usepackage{xspace}
\usepackage{fancyheadings}
\usepackage{layout}
\usepackage{hyperref}

\input xy
\xyoption{matrix}
\xyoption{arrow}
\xyoption{curve}
\newdir{ (}{{}*!/-5pt/@^{(}}

\newboolean{xlabels}
\newcommand{\xlabel}[1]{
                        \label{#1}
                        \ifthenelse{\boolean{xlabels}}
                                   {\marginpar[\hfill{\tiny #1}]{{\tiny #1}}}
                                   {}
                       }

\newcommand{\ZZ}{\mathbb{Z}}

\newcommand{\CC}{\mathbb{C}}

\newcommand{\QQ}{\mathbb{Q}}

\newcommand{\FF}{\mathbb{F}}

\newboolean{probleme}
\newcommand{\problem}[1]
           {\ifthenelse{\boolean{probleme}}
                       {{\bf(PROBLEM: #1)\bf}}
                       {}
           }
\newboolean{zukuenftiges}
\newcommand{\zukunft}[1]
           {\ifthenelse{\boolean{zukuenftiges}}
                       {{\bf(AUSBAUM\"OGLICHKEIT: #1)\bf}}
                       {}
           }
\newboolean{extras}
\newcommand{\extra}[1]
           {\ifthenelse{\boolean{extras}}
                       {{\bf EXTRA #1 EXTRA\bf}}
                       {}
           }

\newboolean{ignore}
\newcommand{\ignore}[1]
           {\ifthenelse{\boolean{ignore}}
                       {{\bf IGNORE #1 IGNORE\bf}}
                       {}
           }

\DeclareMathOperator{\codim}{codim}

\DeclareMathOperator{\spec}{spec}

\DeclareMathOperator{\rad}{rad}

\theoremstyle{plain}
\begingroup

\newtheorem{thm}{Theorem}
\newtheorem{cor}[thm]{Corollary}

\numberwithin{thm}{subsection} 
\endgroup

\newtheorem*{thm*}{Theorem}
\newtheorem*{conj*}{Conjecture}
\newtheorem*{verm*}{Vermutung}

\theoremstyle{definition}

\newtheorem{example}[thm]{Example}

\newtheorem{notation}[thm]{Notation}

\numberwithin{equation}{section}

\newcommand{\nosubsections}{\renewcommand{\thethm}{\thesection.\arabic{thm}}
                            \setcounter{thm}{0}
                           }

\newcommand{\cref}[3]{(\ref{#1}, #2 \ref{#3})}

\date{\today}

\usepackage{amssymb,amsmath}
\setboolean{probleme}{true}
\setboolean{xlabels}{false}

\usepackage{graphicx}

\newcommand{\secemail}{
\setlength{\unitlength}{1pt}
bothmer
\begin{picture}(0,1)
\put(0,0){m}
\put(-5,0){@}
\end{picture}
ath.uni-hannover.de}

\begin{document}

\title{Focal values of plane cubic centers}

\address{Courant Research Centre ''Higher Order Structures''\\
	Mathematisches Institiut\\ 
	University of G\"ottingen\\
         Bunsenstrasse 3-5\\ 
          D-37073 G\"ottingen
         }

\email{\secemail}

\urladdr{http://www.uni-math.gwdg.de/bothmer}

\thanks{Supported by the German Research Foundation 
(Deutsche Forschungsgemeinschaft (DFG)) through 
the Institutional Strategy of the University of G\"ottingen}

\author{Hans-Christian Graf v. Bothmer}
\author{Jakob K\"oker}

\begin{abstract}
We prove that the vanishing of $11$ focal values is not sufficient to ensure that a plane cubic system has a center. 
\end{abstract} 

\maketitle

\newcommand{\dual}{^*}
\newcommand{\barf}{\bar{f}}
\newcommand{\barg}{\bar{g}}
\newcommand{\barh}{\bar{h}}
\newcommand{\bara}{\bar{a}}
\newcommand{\barJ}{\bar{J}}
\newcommand{\barN}{\bar{N}}
\newcommand{\barH}{\bar{H}}

\newcommand{\ZZx}{\ZZ[x_1,\dots,x_n]}
\newcommand{\QQx}{\QQ[x_1,\dots,x_n]}
\newcommand{\CCx}{\CC[x_1,\dots,x_n]}
\newcommand{\FFp}{\FF_p}
\newcommand{\FFx}{\FFp[x_1,\dots,x_n]}
\renewcommand{\AA}{\mathbb{A}}

\newcommand{\zoladek}{\.Zo\l\c adek\,}
\newcommand{\zoladeks}{\.Zo\l\c adek's\,}

\section{Introduction}
\nosubsections
In 1885 Poincar\'e asked when the differential equation
\[
y' = - \frac{x + p(x,y)}{y+q(x,y)} =: - \frac{P(x,y)}{Q(x,y}
\]
with convergent power series $p(x,y)$ and $q(x,y)$ starting with quadratic terms, 
has stable solutions in the neighborhood of the equilibrium solution
$(x,y)=(0,0)$. This means that in such a neighborhood the solutions of the
equivalent plane autonomous system
\begin{align*}
	\dot{x} &= y + q(x,y) = Q(x,y)\\
	\dot{y} &= -x - p(x,y) = -P(x,y)
\end{align*}
are closed curves around $(0,0)$.

Poincar\'e showed that one can iteratively find a formal power series
$F = x^2+y^2+f_3(x,y)+f_4(x,y)+\dots$ such that
\[
	\det \begin{pmatrix} F_x & F_y \\ P & Q \end{pmatrix} = \sum_{j=1}^\infty s_j(x^{2j+2}+y^{2j+2})
\]
with $s_j$ polynomials in the coefficients of $P$ and $Q$.
If all $s_j$ vanish, and $F$ is convergent then $F$ is a constant of motion, i.e. its gradient field
satisfies $Pdx+Qdy=0$. Since $F$ starts with $x^2+y^2$ this shows that close to the origin all integral curves are closed and the system is stable. Therefore the $s_j$'s are
called the {\sl focal values} of $Pdx+Qdy$. Often also the notation $\eta_{2j} := s_j$ is used, and the $\eta_i$ are called {\sl Liapunov quantities}.

Poincar\'e also showed, that if an analytic constant of motion exists, the focal values must vanish.
Later Frommer \cite{Frommer} proved that the systems above are stable if and only if all focal values vanish even without the assumption of convergence of $F$. (Frommer's proof contains a gap which can be closed \cite{vWahlGap})

Unfortunately it is in general impossible to check this condition for a given differential equation because
there are infinitely many focal values. In the case where $P$ and $Q$ are polynomials of degree
at most $d$, the $s_j$ are polynomials in finitely many unknowns. Hilbert's Basis Theorem then implies
that the ideal $I_\infty = (s_1,s_2,\dots)$ is finitely generated, i.e there exists an integer $m := m(d)$
such that
\[
		s_1 = s_2 = \dots = s_{m(d)} = 0  \implies s_j = 0 \quad\forall j.
\]
This shows that a finite criterion for stability exists, but due to the indirect proof of
Hilbert's Basis Theorem no value for $m(d)$ is obtained. In fact even today only $m(2)=3$ is known. 
\zoladek \cite{ZoladekEleven} and Christopher \cite{ChristopherEleven} showed that $m(3) \ge 11$.
Since the number of variables for $d=2$ is six and $m(2)=6-3$ it has 
been conjectured that for $d=3$ with $14$ variables one has $m(3)=14-3=11$.

It is the purpose of this note to prove $m(3) \ge 12$. 

The most naive approach to this problem is to calculate a Gr\"obner Basis of $I_{11}= (s_1,\dots,s_{11})$ and
prove that $s_{12} \not\in I_{11}$ by the usual ideal membership test. Unfortunately this is not feasible, since
the $s_j$ are very complicated. They involve $14$ variables and are of weighted degree $2j$. 
For example $s_5$ has already $5348$ terms and takes about $1.5$ hours on a
Powerbook G4 to calculate. The polynomials $s_j$, $j\ge 6$ can not at the moment be determined by
computer algebra systems.

\zoladek and Christopher therefore deduce their result geometrically. They exhibit a component 
$Y_{11} \subset X_\infty = V(I_\infty)$ that has codimension $11$ in the space of all possible $(P,Q)$ of degree at most three. Finding a component of codimension $12$ is not an easy task, and indeed we choose a different approach. We prove that there exist a codimension $11$ family plane autonomous system of degree $3$ with
a {\sl focus} for which nevertheless the first $11$ focal values vanish, but 12th one doesn't. For this
we look at the system
\begin{align*}
\dot{x} &=y+3x^2 + 8xy + 5y^2+3x^3 + 25x^2y + 20xy^2 + 18y^3\\
\dot{y} &= -(x+27x^2 + 9xy + 22y^2+11x^3 + 20x^2y + 4xy^2 + 3y^3)
\end{align*}
and prove that for this system $s_j = 0 \mod 29$ for $j \le 11$ while $s_{12} \not= 0 \mod 29$. Checking that
furthermore the Jacobian matrix of $s_1,\dots,s_{11}$ has full rank modulo $29$ for this system, we can apply  a theorem of Schreyer \cite{smallFields} to show the existence of the desired family of foci over $\CC$. From this we deduce that $s_{12} \not\in I_{11} = (s_1,\dots,s_{11})$. If fact we even prove the
stronger result $s_{12} \not \in \rad I_{11}$.

Since for given a given system one can evaluate the $s_j$ using Frommers algorithm \cite{martin} without knowing the complete Polynomials, this approach is feasible. 

We found the above system by performing a random search. Heuristically each $s_i$ vanishes mod $29$ for about one of every $29$ differential equations \cite{irred}. So we expect to find an example as above after checking $29^{11} \approx 10^{16}$ random examples. By parametrizing
$s_1$ and $s_2$ we can improve this to $29^9 \approx 10^{13}$ random examples. Indeed we found the example after about $8 \times 10^{12}$ trials. Using an improved version \cite{centerfocusweb} of the program \cite{strudelweb} this took 1246 CPU-days. Since this search is easily parallelizable we could do this calculation in about one month by distributing the work to several computers.

We would like to thank the {\sl Regionales Rechenzentrum f\"ur Niedersachsen (RRZN)}
and the {\sl Institut f\"ur Systems Engineering, Fachgebiet Simulation} for providing the necessary CPU time. Also we are grateful to Colin Christopher who checked our example using REDUCE \cite{reduce}.

\section{The Proof}
\nosubsections

\begin{notation}
If $I \subset \ZZx$ is an ideal and $X_{\ZZ}=V(I) \subset \AA^n_\ZZ$ is the variety over $\spec \ZZ$ defined
by $I$, then we denote by $X_{\FFp}$ the fiber of $X_\ZZ$ over $\FFp$ for any prime $p$. Furthermore we donote by
$X_\CC$ the variety defined by $I$ over $\CC$.
\end{notation}

\begin{thm}[Schreyer] \label{tSchreyer}
Let $I=(f_1,\dots,f_k) \subset \ZZx$ be an ideal and $X_\ZZ = V(I)$. If $x \in X_{\FFp}$ is a point
with $\codim T_{X_{\FFp},x} = k$ then there exists an irreducible component $Y_\ZZ \subset X_\ZZ$ with
$x \in Y_\ZZ$ and $Y_\ZZ \not\subset X_{\FFp}$. In particular $Y_\CC \not= \emptyset$
\end{thm}

\begin{proof} This is a special case of a theorem of Schreyer \cite{smallFields}. See also \cite{newFamily} for a proof.
\end{proof}

\begin{figure}[h] 
\includegraphics*[width=10cm]{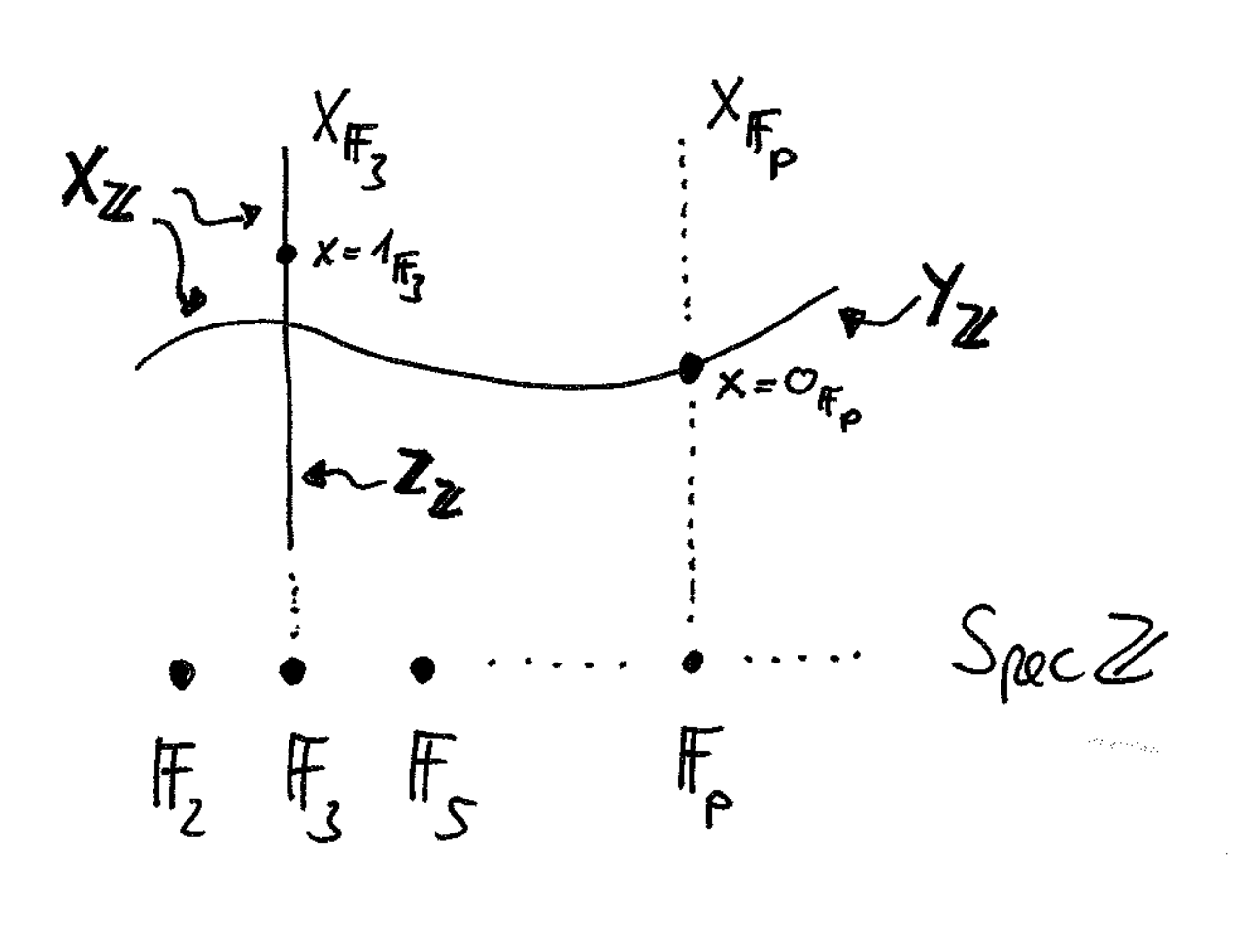}
\caption{A variety over $\spec \ZZ$} \label{fSpecZZ}

\end{figure}

\begin{example}
Consider $X_\ZZ = V(3x) \subset \AA_\ZZ^1$. This variety has two components over $\ZZ$ namely
$Y_\ZZ=V(x)$ and $Z_\ZZ = V(3)$. Since $3=0$ is true only in $\FF_3$ we have $Z_\ZZ = Z_{\FF_3}$. Furthermore $Z_\CC = \emptyset$. On the other hand $x=0$ is possible over all $\FFp$ and $Y_\CC \not=\emptyset$. See Figure \ref{fSpecZZ}.

Indeed, if we consider the point $x=0\in X_{\FFp}$, $p\not=3$ then we have that the derivative $(3x)'=3 \not=0$ and the tangent space $T_{0,X_{\FF_p}}$ has codimension $1$. Therefore the Theorem applies and the component $Y_\ZZ$ containig $x=0_{\FFp}$ is not contained in $X_{\FFp}$.

Since $3\cdot 1 = 0 \in \FF_3$ we can also consider the point $x=1 \in X_{\FF_3}$. Here we have $(3x)'=3=0$ and 
the tangent space $T_{1,X_{\FF_3}}$ has codimension $0$. Hence the Theorem does not apply, and indeed
the component $Z_\ZZ = Z_{\FF_3}$ containing $x=1_{\FF_3}$ is completely contained in $X_{\FF_3}$.
\end{example}

\begin{cor} \label{cNotVanish}
If in the situation of Theorem \ref{tSchreyer} we have a further polynomial $g \in \ZZx$ satisfying 
$g(x) \not=0 \in \FFp$ then $g$ does not vanish on $X_\CC$.
\end{cor}

\begin{proof}
Assume to the contrary that $g$ vanishes on $X_\CC$. 
By Theorem \ref{tSchreyer} we have a component $Y_\ZZ \subset X_\ZZ$ with $x\in Y_\ZZ$ and $Y_\CC \not= \emptyset$. Since
$g$ vanishes on $X_\CC$ and $Y_\CC \not= \emptyset$ is also vanishes on $Y_\CC$ and therefore on $Y_\ZZ$ and
$Y_{\FFp}$. But this contradicts our assumption $g(x) \not=0$.
\end{proof}

\begin{thm}
$m(3) \ge 12$.
\end{thm}

\begin{proof}
Use our implementation of Frommers algorithm \cite{martin}, \cite{strudelweb}, \cite{centerfocusweb} or REDUCE \cite{reduce} to check that the example in the introduction satisfies the conditions of 
Corollary \ref{cNotVanish}.
\end{proof}

\def\cprime{$'$} \def\cprime{$'$}

\end{document}